\documentclass[12pt,a4paper]{amsart}

\usepackage[utf8]{inputenc}
\usepackage{comment}
\usepackage[T1]{fontenc}
\usepackage{mathtools}
\usepackage[a4paper, margin=2.7cm]{geometry}

\usepackage{amsmath, amsthm, amsfonts, amssymb,bm}
\usepackage{bbm}
\usepackage{mathrsfs}  
\usepackage{enumitem}
\usepackage[colorlinks=true,linkcolor=blue,citecolor=blue,urlcolor=blue,breaklinks]{hyperref}

\usepackage{graphicx}
\usepackage{mathabx}
\usepackage[toc,page]{appendix}
\usepackage{hyperref}
\usepackage{cite}
\usepackage{bigints}
\usepackage[dvipsnames]{xcolor}

\usepackage[toc,page]{appendix}

\newtheorem{theorem}{Theorem}[section]
\newtheorem{proposition}[theorem]{Proposition}
\newtheorem{lemma}[theorem]{Lemma}

\newtheorem{remark}{Remark}[section]

\theoremstyle{definition}
\newtheorem{definition}[theorem]{Definition}

\newcommand{\R}{\mathbb{R}}
\newcommand{\C}{\mathbb{C}}

\newcommand{\M}{\mathcal{M}}
\newcommand{\N}{\mathbb{N}}

\newcommand{\D}{\mathcal D}

\newcommand{\an}[1]{\left\langle #1 \right\rangle}

\newcommand{\Man}{\mathcal{M}}

\newcommand{\vertiii}[1]{{\left\vert\kern-0.25ex\left\vert\kern-0.25ex\left\vert #1 
    \right\vert\kern-0.25ex\right\vert\kern-0.25ex\right\vert}}

\title[Strichartz estimates for the Dirac equation]{Strichartz estimates for the Dirac equation on asymptotically flat manifolds}

\begin{document}

\author{Federico Cacciafesta}
\address{Federico Cacciafesta: 
Dipartimento di Matematica, Universit\'a degli studi di Padova, Via Trieste, 63, 35131 Padova PD, Italy}
\email{cacciafe@math.unipd.it}

\author{Anne-Sophie de Suzzoni}
\address{Anne-Sophie de Suzzoni:
CMLS, \'Ecole  Polytechnique, CNRS, Universit\'e Paris-Saclay, 91128 PALAISEAU Cedex, France}
\email{anne-sophie.de-suzzoni@polytechnique.edu}

\author{Long Meng}
\address{Long Meng:
CERMICS, \'Ecole des ponts ParisTech, 6 and 8 av. Pascal, 77455 Marne-la-Vall\'ee, France}
\email{long.meng@enpc.fr}


\maketitle

\begin{abstract}
In this paper we prove Strichartz estimates for the Dirac equation on asymptotically flat manifolds. The proof combines the weak dispersive estimates proved in \cite{Weakdis} with the Strichartz and smoothing estimates for the wave and Klein-Gordon flows exploiting the results in \cite{sogge2010concerning} and \cite{d2015kato}-\cite{zhang2019strichartz} in the same geometrical setting. 
\end{abstract}

\section{Introduction}


 In \cite{Weakdis}, the authors started the study of dispersive dynamics of the Dirac equation in a non-flat setting, proving local smoothing estimates in the cases of asymptotically flat manifolds and warped products. Later on, in \cite{cacdes2} and \cite{cacdes3}, the authors proved respectively local and global in time weighted Strichartz estimates for spherically symmetric manifolds: in this case, it is indeed possible to take advantage of the so called {\em partial wave decomposition}, which is analog of the spherical harmonics decomposition for the Dirac operator, in order to recast the equation into a sum of ``radial Dirac equations'' with potentials, for which several results are available. The problem of proving Strichartz estimates {\em without} the assumption of spherical symmetry on the manifold could not be solved directly, as indeed it was not possible to apply the standard Duhamel trick to deduce them from the ones on $\mathbb{R}^3$, since the equation in this framework is a first order perturbation of the flat one. The purpose of this short paper is to fill this gap, proving in fact Strichartz estimates for the Dirac equation on asymptotically flat manifolds of dimension 3.

We thus consider the following Cauchy problem
\begin{align}\label{eq:ques-D}
    \left\{
    \begin{aligned}
    &i\partial_t u -\D_m u=0\\
    &u(0,x)=u_0(x),\quad x\in\R^3.
    \end{aligned}
    \right.
\end{align}
where $u:\Man\rightarrow \mathbb{C}^4$, $(\Man,g)$ is a manifold with a Lorentzian metrics $g$ that decouples space and time. In other words, we assume that it writes
\begin{equation}\label{struct}
g_{jk} = \left \lbrace{\begin{array}{ll}
1 & \textrm{ if } j= k = 0\\
0 & \textrm{ if } jk = 0 \textrm{ and } j\neq k\\
-h_{jk}(x) & \textrm{ otherwise. }\end{array}} \right.
\end{equation}
endowed with a spin structure. We refer to Section 2 in \cite{Weakdis} (see also \cite{parktoms}, section 5.6) for all the details on the construction and properties of the Dirac operator on a manifold $\Man$: here, we very briefly limit ourselves to recall that the Dirac operator with mass $m\geq0$ can be written as
\begin{align}
    \D_m=-i\gamma^0\gamma^ae^i_{\; a}D_i-\gamma^0m.
\end{align}
where $\gamma^j$ denote the standard Dirac matrices, that is $\gamma^0=\beta$ and $\gamma^j=\gamma^0\alpha_j$ for $j=1,2,3$ with 
\begin{equation}\label{diracmatrices}
\alpha_j=\left(\begin{array}{cc}0 & \sigma_j \\\sigma_j & 0\end{array}\right),\quad \beta=
\left(\begin{array}{cc}I_2 & 0 \\0 & -I_2\end{array}\right)
\end{equation}
and $\sigma_j$ are the Pauli matrices
\begin{equation}
\sigma_1=\left(\begin{array}{cc}0 & 1 \\1 & 0\end{array}\right),\quad
\sigma_2=\left(\begin{array}{cc}0 &
-i \\i & 0\end{array}\right),\quad
\sigma_3=\left(\begin{array}{cc}1 & 0\\0 & -1\end{array}\right),
\end{equation}
$e^i_{\; a}$ is a matrix bundle satisfying 
\begin{equation}\label{dre}
    h^{ij} = e^i_{\; a}\delta^{ab}e^j_{\; b}
\end{equation}
where $\delta$ is the Kronecker symbol. We call it a {\em dreibein} since it is restriction to space of a {\em vierbein}. A vierbein is a matrix bundle $\tilde e^\mu_{\; a}$ satisfying 
\[
\tilde e^\mu_{\; a}\eta^{ab}\tilde e^\nu_{\; b} = g^{\mu\nu};
\]
setting $e^j_{\; a} = \tilde e^j_{\; a}$ for $j,a = 1,2,3$, we get indeed \eqref{dre}. Note that the existence of such a dreibein is induced by the asymptotic flatness of the manifold. The covariant derivative $D_i$ is given by
\begin{align}\label{conneq}
    D_0=\partial_0,\quad D_j=\partial_j+B_j,\quad j=1,2,3
\end{align}
where $B_j$ writes
\[
B_j = \frac{1}{8} [\gamma_a,\gamma_b] \Omega_j^{\; ab}.
\]
It has a purely geometric part $\omega_j^{\; ab}$, called the {\em spin connection} 
\begin{equation}\label{spincon}
    \omega_j^{\;ab}=e^i_{\; a}\partial_j e^{ib}+e_i^{\;a}\Gamma_{\; jk}^ie^{kb}
\end{equation}
with the Christoffel symbol (or affine connection) $\Gamma^i_{\; jk}$ given by
\begin{equation}\label{ichtus}
\Gamma_{\; jk}^i:=\frac{1}{2}h^{il}(\partial_j h_{lk}+\partial_kh_{jl}-\partial_l h_{jk}),
\end{equation}
and a purely algebraic part $\frac18[\gamma_a,\gamma_b]$, which is due to the nature of the particle we consider (here a pair electron-positron) and more specifically to its spin (here $\frac12\oplus \frac12$).

Finally we recall that the \textit{scalar curvature} $\mathcal{R}_h$ writes
\begin{equation}\label{scacur}
\mathcal{R}_h:=h^{jk}\left(\partial_i\Gamma_{\; jk}^i-\partial_k\Gamma_{\; ji}^i+\Gamma_{\;jk}^l\Gamma^i_{\; il}-\Gamma^l_{\; ji}\Gamma_{\; kl}^i\right).
\end{equation}

\medskip

For what concerns the manifold $\Man$, we assume the following

\medskip

{\bf Assumptions (A).} Let $(\Man,g)$ be a 4-dimensional Lorentzian manifold with a metrics $g$ having the structure given by \eqref{struct}, and $h\in C^\infty(\mathbb{R}^3)$.
We assume that there exists a constant $C_h$ and $\sigma \in (0,1)$ such that for all $\alpha \in \N^3$ such that $|\alpha| = \alpha_1+\alpha_2+\alpha_3\leq 3$ and all $x$,
\begin{equation}\label{assasy}
 |\partial^\alpha(h_{ij}(x)-\delta_{ij})|\leq  C_h\an{x}^{-|\alpha|-1-\sigma}
\end{equation}
where $\partial^\alpha = \partial_1^{\alpha_1}\partial_2^{\alpha_2}\partial_3^{\alpha_3}$. 

\begin{remark}
The assumptions above are fairly standard, and manifolds satisfying these are usually referred to as asymptotically flat manifolds. We stress the fact that requiring the constant $C_h$ to be small enough is a sufficient condition to ensure that the manifold is  non trapping (see e.g. \cite{tataru}, \cite{cacciafestadanconaluca}): thus, we will not have to assume this condition, that is crucial in order to have dispersion. Besides, it is easy to see that condition \eqref{assasy} holds for the inverse matrix of $h$ as well, provided the constant $C_h$ is sufficiently small. Finally, let us mention the fact that the decay condition \eqref{assasy} might not be optimal; in particular, the power $-|\alpha|-1-\sigma$  could be weakened to $-|\alpha|-\sigma$ in the massless case, but we here prefer to provide a unified and much simpler presentation of the results.
\end{remark}

\begin{remark}\label{rem:fBR}

According to \eqref{dre}, one can bound, in the sense of matrices, the square of $e$ with $h$. As $h$ is ``close'' to the identity, estimate \eqref{assasy} holds true for the matrices $e$. Thus, under assumptions {\bf (A)}, it is possible to prove that there exist constants $C_B,\,C_B',\,C_{\Gamma},\,C_R>0$ such that
\[
|B|\leq C_BC_h\an{x}^{-2-\sigma},\quad |\partial B|\leq C_B'C_h\an{x}^{-2-\sigma}
\]
\[
|\mathcal{R}_h|\leq C_RC_h\an{x}^{-3-\sigma},\quad |\Gamma|\leq C_{\Gamma}C_h\an{x}^{-2-\sigma}.
\]
These bounds will be proved in forthcoming Proposition \ref{propbound}.
\end{remark}

Before stating our Theorem, let us recall the definition of admissible Strichartz triple:

\begin{definition}
 In dimension $3$, the triple $(s,q,r)$ is called {\em wave admissible} if 
\[
\frac{1}{q}= \frac{1}{2}-\frac{1}{r},\quad 2\leq q,r\leq \infty,\quad  r\neq\infty \quad s=\frac{1}{2}-\frac{1}{r}+\frac{1}{q}.
\]
The triple $(s,q,r)$ is called {\em Klein-Gordon} (or {\em Schr\"odinger}) {\em admissible} if
\[
\frac{2}{q}= \frac{3}{2}-\frac{3}{r},\quad 2\leq q\leq \infty,\quad  2\leq r\leq 6 \quad s=\frac{1}{2}-\frac{1}{r}+\frac{1}{q}.
\]
\end{definition}

We will use the standard notation $L^p$ for Lebesgue spaces; more precisely, the norms in time will be denoted by $L^p_t$ and will be intended to be on $\mathbb{R}^+$, and we define
\begin{equation*}
  \|f\|_{L^p(\Man_h)}^p= \int_{\mathbb{R}^3}|f(x)|^p \sqrt{\det h(x)}dx.
\end{equation*}
Then we define the $H^s_p$ and $\dot H^s_p$ norms for $s\in \N$ as
\begin{equation*}
    \|f\|_{\dot H^s_p(\Man_h)}: =\|(-\widetilde{\Delta}_h)^{s/2}f\|_{L^p(\Man_h)},  \quad \|f\|_{ H^s_p(\Man_h)} :=\|(1-\widetilde{\Delta}_h)^{s/2}f\|_{L^p(\Man_h)}
\end{equation*}
where $\widetilde{\Delta}_h$ is the standard Laplace-Beltrami operator. For negative $s$, we define these spaces by duality and for fractional $s$, we define these spaces by interpolation.

\medskip

Our main result is then the following:
\begin{theorem}\label{teo1}
Let $(\Man,g)$ be as given by assumptions {\bf (A)}. Then the massless Dirac flow satisfies the Strichartz estimate:
\begin{align}\label{massless}
    \|e^{it\D}u_0\|_{L^q_t \dot{H}^{1-s}_r(\Man_h)}\lesssim\|u_0\|_{\dot{H}^1(\Man_h)}
\end{align}
for all wave admissible triple $(s,q,r)$, while in the massive case we have
\begin{align}\label{massive}
    \|e^{it\D_m}u_0\|_{L^q_tH^{1/2-s}_r(\Man_h)}\lesssim\|u_0\|_{H^1(\Man_h)}\quad (m\neq 0)
\end{align}
for all Klein-Gordon admissible triple  $(s,q,r)$ such that $q>2$.
\end{theorem}



Let us briefly comment on the strategy of our proof, which is short, but relies on several different recent results. The idea consists in squaring equation \eqref{eq:ques-D} in order to obtain a system of wave or Klein-Gordon equations on the manifold $(\Man,g)$ (depending on whether $m=0$ or $m>0$), and then combining the estimates for such flows on asymptotically flat manifolds with the standard argument based on Duhamel formula and local smoothing estimates to control the ``perturbative terms''. This trick is in fact widely used for the study of several properties of the Dirac equation, as the Dirac operator is indeed built as a suitable ``square root'' of the Laplacian. Anyway, we should stress the fact that this strategy comes with two main difficulties: one is in that the Laplace-Beltrami operator obtained after the squaring procedure is {\em not} the standard (or ``scalar'') one, but it is a ``spinorial'' Laplace operator (the covariant derivative is not the same as the covariant derivatives for scalar or vector fields). Therefore, it will not be possible to apply directly the existing results, and we will have somehow to estimate the difference of the solutions to the ``scalar'' and the ``spinorial'' wave/Klein-Gordon equations. This difference contains a first order term, and this represents the second difficulty, as indeed we will have to rely on a local smoothing at the ``first order'' level.

\medskip

{\bf Acknowledgments.} F.C. and L.M. acknowledge support from the University of Padova STARS project ``Linear and Nonlinear Problems for the Dirac Equation'' (LANPDE), and AS. dS. is supported by the ANR project ESSED ANR-18-CE40-0028.

\section{Proof of Theorem \ref{teo1}}

We begin with proving the statements in Remark \ref{rem:fBR}.

\begin{proposition}\label{propbound} Assume that $C_h\ll 1$. The dreibein $e$ exists and can be chosen such that there exist constants $C_B$ and $C_B'$ such that for all $x$, we have 
\[
|B(x)|\leq C_B C_h \an{x}^{-\sigma- 2},\quad |\partial B(x)|\leq C_B' C_h \an{x}^{-\sigma-3}.
\]
What is more, there exist constants $C_R$ and $C_\Gamma$ such that for all $x$, we have
\[
 |\mathcal R_h (x)| \leq C_R C_h \an{x}^{-\sigma-3},\quad  |\Gamma| \leq C_\Gamma C_h \an{x}^{-\sigma-2}.
\]
\end{proposition}

\begin{proof} The scalar curvature does not depend on $e$ and writes
\[
\mathcal R_h = h^{jk}(\partial_i\Gamma^i_{\;jk} - \partial_k \Gamma^i_{\; ji}+\Gamma^{l}_{\;jk}\Gamma^i_{\;li} - \Gamma^l_{\;ij}\Gamma^i_{kl})
\]
with the affine connection given by \eqref{ichtus}.
Therefore for all $x$, we have 
\[
|\mathcal R_h (x)|\leq C_R' (|\partial \Gamma (x)| + |\Gamma(x)|^2).
\]
Indeed, we can choose $C_h$ small enough such that for all $x$, $|h^{-1}(x)|\leq 2$.
Therefore,
\[
|\Gamma(x)| \leq C_\Gamma |h'(x)| \leq C_\Gamma C_h \an{x}^{-2-\sigma}
\]
and since $(h^{-1})'(x) = - h^{-1}(x)h'(x)h^{-1}(x)$, choosing $C_h\ll 1$, we get
\[
|\partial \Gamma (x)|\leq C_\Gamma' (|h'(x)|^2 + |h''(x)|) \leq C_\Gamma' C_h \an{x}^{-3-\sigma}.
\]
We deduce 
\[
|\mathcal R_h(x)|\leq C_R C_h \an{x}^{-3-\sigma}.
\]
We look for $e$ a matrix bundle such that 
\[
h^{ij}(x) = e^i_{\; a}(x) \delta^{ab} e^j_{\; b}(x)
\]
for all $x$. This can be rewritten as 
\[
h^{ij}(x) = e^{ia}(x) \delta_{ab}e^{jb}(x).
\]
If we restrict $(e^{ia})_{1\leq i,a\leq 3}$ to be symmetric, this rewrites as
\[
h=e^2.
\]
As it is well-known, $e\mapsto e^2$ seen as a map from the the symmetric matrices to the symmetric matrices is $\mathcal C^\infty$ and its differential at the identity is twice the identity of the symmetric matrices. Therefore, it can be reversed into a $\mathcal C^\infty$ map $F$ (a square root) around the identity since $Id^2 = Id$. We choose $C_h$ small enough such that for all $x$, $h(x)$ lies in $K$ a compact subset of the definition set of $F$ and we choose $e(x) = F(h(x))$ for all $x$. By the increment theorem, we have 
\[
|e(x) - Id| \leq \sup_K |DF| |h(x)-Id|.
\]
We also have that for all $x$,
\begin{multline*}
|e'(x)| \leq \sup_K |D F|\, |h'(x)| , \quad |e''(x)| \leq \sup_K |D^2 F|\, |h'(x)|^2 + \sup_K |DF|\, |h''(x)|, \\
|e'''(x)|\leq \sup_K |D^3 F|\, |h'(x)|^3 + 3 \sup_K|D^2F| \, |h'(x)|\, |h''(x)| + \sup_K |D F|\, |h'''(x)|.
\end{multline*}
Therefore, we deduce that for all $x$,
\begin{multline*}
|e(x)-Id| \leq C_e C_h \an{x}^{-\sigma-1}, \quad |e'(x)|\leq C_e C_h \an{x}^{-2-\sigma},\\
|e''(x)|\leq C_e C_h \an{x}^{-3-\sigma}, \quad |e'''(x)|\leq C_e C_h \an{x}^{-4-\sigma}.
\end{multline*}
We deduce 
\[
|\omega (x)|\leq C_\omega C_h \an{x}^{-2-\sigma}, \quad |B(x)| \leq  C_B  C_h \an{x}^{-2-\sigma},\quad \textrm{and}\quad |\partial B(x)| \leq  C_B'  C_h \an{x}^{-3-\sigma}.
\]
\end{proof}

We now recall the connection between the Dirac and the wave/Klein-Gordon equations on manifolds, and recall the local smoothing estimate for it. This is the first main ingredient in our proof.

\begin{theorem}\label{lem:Dirac-wave}
If $u$ is a solution to \eqref{eq:ques-D}, then it also solves the following
\begin{align}\label{eq:ques-KG}
    \left\{
    \begin{aligned}
    &\partial_t^2u+m^2u-\Delta_h u+\frac{1}{4}\mathcal{R}_h u=0,\quad (t,x)\in \R^+\times \R^3,\\
    &\partial_t u(0,x)=i\D_m u_0(x),\quad u(0,x)=u_0(x),\quad x\in\R^3
    \end{aligned}
    \right.
\end{align}
where $\Delta_h =D^i D_i$ with $D_i$ given by \eqref{conneq}. Besides, 
\begin{align}\label{spinvsscal}
    \Delta_h v=\widetilde{\Delta}_h v+ B^i\partial_i v+\widetilde{D}^i B_iv+B^iB_i v
\end{align}
where $\widetilde{\Delta}_h$ is the Laplace-Beltrami operator for scalars tensor the 4-dimensional identity matrix, $\widetilde{D}^i\Psi_k=\partial^i\Psi_k-\Gamma^{l\; i}_{\; k}\Psi_l$ and $B^i=h^{ij}B_j$. 

Let $u$ be a solution of \eqref{eq:ques-D}, then
\begin{align}\label{loc-3}
   \|\an{x}^{-3/2-}u\|_{L^2_tL^2(\Man_h)} + \|\an{x}^{-1/2-}\widetilde{\nabla} u\|_{L^2_tL^2(\Man_h)}\lesssim \|\D u_0\|_{L^2(\Man_h)}
\end{align}
where $\widetilde{\nabla}$ denotes the gradient for scalar fields, that is
\[
\|\an{x}^{-1/2-}\widetilde{\nabla} u\|_{L^2_tL^2(\Man_h)}^2 := \int_{\Man} \an{x}^{-1-} h^{ij}\an{\partial_i u , \partial_j u}_{\C^4} dx.
\]
\end{theorem}

\begin{proof}
The relation between the Dirac equation and system \eqref{eq:ques-KG} is shown in Corollary 1 and Formula (34) in \cite{Weakdis}. Furthermore, by Proposition 4 in \cite{Weakdis}, we have
\begin{align}\label{eq:D^2}
    \D_m^2=m^2-\Delta_h+\frac{1}{4}\mathcal{R}_h.
\end{align}

Let us now prove the local smoothing estimate \eqref{loc-3}. It is shown in \cite[Theorem 1.1]{Weakdis} that
\begin{align}\label{loc-4}
   \|\an{x}^{-3/2-}u\|_{L^2_tL^2(\Man_h)} + \|\an{x}^{-1/2-}\nabla u\|_{L^2_tL^2(\Man_h)}\lesssim \|\D_m u_0\|_{L^2(\Man_h)}.
\end{align}
Recall that $\nabla=(D_1,D_2,D_3)=\widetilde{\nabla}+B$ where $B=(B_1,B_2,B_3)$. According to Proposition \ref{propbound}, 
\[
\|\an{x}^{-1/2-}Bu\|_{L^2_tL^2(\Man_h)}\leq C_BC_h\|\an{x}^{-3/2-}u\|_{L^2_tL^2(\Man_h)}\lesssim \|\D_m u_0\|_{L^2(\Man_h)}.
\]
Thus,
\begin{align*}
    \MoveEqLeft \|\an{x}^{-3/2-}u\|_{L^2_tL^2(\Man_h)} + \|\an{x}^{-1/2-}\widetilde{\nabla} u\|_{L^2_tL^2(\Man_h)}\\
    &\leq \|\an{x}^{-3/2-}u\|_{L^2_tL^2(\Man_h)}+ \|\an{x}^{-1/2-}\nabla u\|_{L^2_tL^2(\Man_h)}+ \|\an{x}^{-1/2-}B u\|_{L^2_tL^2(\Man_h)}.
\end{align*}
Hence 
\[
\|\an{x}^{-3/2-}u\|_{L^2_tL^2(\Man_h)} + \|\an{x}^{-1/2-}\widetilde{\nabla} u\|_{L^2_tL^2(\Man_h)}\lesssim\|\D_m u_0\|_{L^2(\Man_h)}
\]
and this concludes the proof.
\end{proof}

As a second ingredient, we need to provide suitable Strichartz and local smoothing estimates for the solutions to the ``auxiliary'' systems of wave and Klein-Gordon equations. We state them in the next two Theorems.

\begin{theorem}[Strichartz estimates for wave/Klein-Gordon]\label{th:Stri-wave}
Let $(\Man,g)$ be a 4-dimensional Lorentzian manifold satisfying assumptions {\bf (A)} and let $u$ be a solution to the following system:
\begin{align}\label{eq:ques-wave}
    \left\{
    \begin{aligned}
    &\partial_t^2u+m^2u-\widetilde{\Delta}_h u=0,\quad (t,x)\in {(t,x)\in \Man},\\
    &\partial_t u(0,x)=u_1(x),\quad u(0,x)=u_0(x),\quad x\in\R^3.
    \end{aligned}
    \right.
\end{align}
If $m=0$, then u satisfies
\begin{align}\label{stri-w}
    \|u\|_{L^q_t\dot{H}^{1-s}_{r}(\Man_h)}\lesssim \|u_0\|_{\dot{H}^{1}(\Man_h)}+\|u_1\|_{L^2(\Man_h)}
\end{align}
for any wave admissible triple $(s,q,r)$. 

\noindent
If $m> 0$, then $u$ satisfies
\begin{align}\label{stri-KG}
    \|u\|_{L^q_t H^{1/2-s}_{r}(\Man_h)}\lesssim \|u_0\|_{\dot{H}^{1/2}(\Man_h)}+\|u_1\|_{H^{-1/2}(\Man_h)}.
\end{align}
for any Klein-Gordon admissible triple $(s,q,r)$.
\end{theorem}
\begin{proof}
When $m=0$, this is just Theorem 1.4 in \cite{sogge2010concerning}. When $m> 0$, the Strichartz estimate follows from the global-in-time Strichartz estimate on non-trapping conic manifold (i.e. scattering manifold) in \cite[Theorem 1.1]{zhang2019strichartz}. It is shown in \cite[Remark 1.2]{hassell2005strichartz} that any asymptotically flat space $(\mathbb{R}^3,h)$ with decay estimates $|\partial^\alpha(h-\delta)|\leq C_j\an{x}^{-|\alpha|-1}$ is also asymptotically conic (see also \cite[Remark 1.5]{rodnianski2015effective}). Hence we can deduce the result in the case $m> 0$.
\end{proof}

\begin{theorem}[Local smoothing estimates for wave/Klein-Gordon]\label{th:loc-wave}
Let $(\Man,g)$ be a 4-dimensional Lorentzian manifold satisfying assumptions {\bf (A)}. Then the following estimates hold
\begin{align}\label{eq:loc-wave-1}
    \|\an{x}^{-1/2-}e^{it\sqrt{-\widetilde{\Delta}_h}}f\|_{L^2_tL^2(\Man_h)}\lesssim \|f\|_{L^2(\Man_h)},
\end{align}
for any $f\in L^2(\Man_h)$, and
\begin{align}\label{eq:loc-wave-3}
      \|\an{x}^{-1/2}e^{it\sqrt{m^2-\widetilde{\Delta}_h}}f\|_{L^2_tL^2(\Man_h)}\lesssim \|(1-\widetilde \Delta_h)^{1/4}f\|_{L^2(\Man_h)}
\end{align}
for any $f$ such that $(1-\widetilde \Delta_h)^{1/4}f \in L^2(\Man_h)$.
\end{theorem}
\begin{proof}
Let us consider the following unitary transform
\begin{align*}
    \mathcal{V}: L^2(\mathbb{R}^3,\sqrt{\det h(x)}dx)\to L^2(\mathbb{R}^3,dx),\quad v\mapsto (\det h(x))^{1/4}v.
\end{align*}
The transformation $\mathcal{V}$ sends $-\widetilde{\Delta}_h$ to 
\[
P=-\mathcal{V}\widetilde{\Delta}_h \mathcal{V}^{-1}=-(\det h(x))^{1/4}\widetilde{\Delta}_h(\det h(x))^{-1/4}.
\]
Let us start with the massless case. Let $u:=e^{it\sqrt{-\widetilde{\Delta}_h}}f\in L^2(\Man,dg)$ and let $v=\mathcal{V}u$; then $v\in L^2(\mathbb{R}^+_t\times\mathbb{R}^3,dx)$, and
\begin{align*}
\partial_t^2u-\widetilde{\Delta}_hu=0 \Leftrightarrow\partial_t^2v-P v=0.
\end{align*}
According to 
\cite[Theorem 1.3]{BonySemi} or \cite[Page 24, Section 6]{sogge2010concerning}, we get
\begin{align*}
    \|\an{x}^{-1/2-}u\|_{L^2_tL^2(\Man_h)}=\|\an{x}^{-1/2-}v\|_{L^2(\mathbb{R}^+\times\mathbb{R}^3)}\lesssim \|\an{x}^{-1/2-}v(0)\|_{L^2(\mathbb{R}^3)}=\|f\|_{L^2(\Man_h)}.
\end{align*}
For the second estimate (massive case), according to \cite[Formula (3.5) and Proposition 3.1]{zhang2017global} (taking $V=0$), we know that $\an{x}^{-1}$ is $-\widetilde{\Delta}_h$-smooth, i.e.,
\[
\|\an{x}^{-1}e^{-it\widetilde{\Delta}_h}f\|_{L^2_tL^2(\Man_h)}\lesssim \|f\|_{L^2(\Man_h)}.
\]
It follows from \cite[Theorem 2.2 and Theorem 2.4]{d2015kato} that
\[
\|\an{x}^{-1/2}e^{it(-\sqrt{m^2-\widetilde{\Delta}})}f\|_{L^2_tL^2(\Man_h)}\lesssim \|(m^2-\widetilde{\Delta}_h)^{1/4}f\|_{L^2}
\]
and this concludes the proof.
\end{proof}

We are now in a position to prove our main result.
\begin{proof}[Proof of Theorem \ref{teo1}.]
According to \eqref{eq:ques-KG} and \eqref{spinvsscal}, by Duhamel Formula, we can write for any $m\geq0$
\begin{align}
    u(t,x):=e^{it\D_m}u_0=\dot{W}_m(t)u_0+iW_m(t)\D_m u_0+\int_0^t W_m(t-s)(\Omega_1(u)(s)+\Omega_2 u(s)) ds
\end{align}
where 
\[
W_m(t)=\frac{\sin(t\sqrt{m^2-\widetilde{\Delta}_h})}{\sqrt{m^2-\widetilde{\Delta}_h}},\quad \dot{W}_m=\partial_t W_m
\]
and 
\begin{equation}\label{Omega}
\Omega_1(u):=2B^i\partial_i u,\quad \Omega_2:=\partial^iB_i+B^iB_i-\Gamma_{\; i}^{j\; i}B_j-\frac14 \mathcal{R}_h.
\end{equation}

We deal with the massless and massive cases separately, starting with the former. According to Christ-Kiselev Lemma \cite{ChrKisMax} and Theorem \ref{th:Stri-wave}, 
\begin{multline*}
  \left\|\int_0^TW_m(t-s)(\Omega_1(u)(s)+\Omega_2 u(s))ds\right\|_{L^q_t\dot{H}^{1-s}_r(\Man_h)}\\
  \lesssim \left\|\frac{e^{it\sqrt{-\widetilde{\Delta}_h}}}{\sqrt{-\widetilde{\Delta}_h}}\int_0^Te^{-is\sqrt{-\widetilde{\Delta}_h}} (\Omega_1(u)(s)+\Omega_2 u(s))ds\right\|_{L^q_t\dot{H}^{1-s}_r(\Man_h)}\\
  \lesssim \left\|\int_0^Te^{-is\sqrt{-\widetilde{\Delta}_h}} (\Omega_1(u)(s)u(s)+\Omega_2 u(s))ds\right\|_{L^2(\Man_h)}.
\end{multline*}
By the dual form of \eqref{eq:loc-wave-1}, 
\begin{multline*}
    \left\|\int_0^Te^{-is\sqrt{-\widetilde{\Delta}_h}} (\Omega_1(u)(s)u(s)+\Omega_2 u(s))ds\right\|_{L^2(\Man_h)}
    \lesssim \|\an{x}^{1/2+}(\Omega_1(u)(s)u+\Omega_2 u)\|_{L^2_tL^2(\Man_h)}.
\end{multline*}
Then by \eqref{loc-3} and Remark \ref{rem:fBR}, we have
\[
\|\an{x}^{1/2+}\Omega_2 u\|_{L^2_tL^2(\Man_h)}\lesssim \|\an{x}^{2+}\Omega_2\|_{L^\infty}\|\an{x}^{-3/2-}u\|_{L^2_tL^2(\Man_h)}\lesssim \|\D u_0\|_{L^2(\Man_h)}.
\]

Then by \eqref{loc-3} and Remark \ref{rem:fBR}, we have
\[
\|\an{x}^{1/2+}\Omega_1(u)\|_{L^2_tL^2(\Man_h)}\lesssim \|\an{x}^{1+}B\|_{L^\infty_x}\|\an{x}^{-1/2-}\widetilde{\nabla} u\|_{L^2_tL^2(\Man_h)}\lesssim \|\D u_0\|_{L^2(\Man_h)}.
\]
Thus thanks to Lemma \ref{lem:norm},
\[
\|u\|_{L^q_t\dot{H}^{1-s}_r(\Man_h)}\lesssim \|u\|_{\dot{H}^1(\Man_h)}+\|\D u_0\|_{L^2(\Man_h)}\lesssim \|u\|_{\dot{H}^1(\Man_h)}.
\]

\medskip
Now we consider the massive case. According to Christ-Kiselev lemma \cite{ChrKisMax}, since we restrict to the case $q>2$, and Theorem \ref{th:Stri-wave}, 
\begin{multline*}
  \left\|\int_0^TW_m(t-s)(\Omega_1(u)(s)+\Omega_2 u(s))ds\right\|_{L^q_t H^{1/2-s}_r(\Man_h)}\\
  \leq \left\|\frac{e^{it\sqrt{m^2-\widetilde{\Delta}_h}}}{\sqrt{m^2-\widetilde{\Delta}_h}}\int_0^Te^{-is\sqrt{m^2-\widetilde{\Delta}_h}} (\Omega_1(u)(s)+\Omega_2 u(s))ds\right\|_{L^q_tH^{1/2-s}_r(\Man_h)}\\
  \leq \left\|\int_0^Te^{-is\sqrt{m^2-\widetilde{\Delta}_h}} (\Omega_1(u)(s)+\Omega_2 u(s))ds\right\|_{H^{-1/2}(\Man_h)}.
\end{multline*}
By the dual form of \eqref{eq:loc-wave-3}, 
\[
\left\|\int_0^Te^{-is\sqrt{m^2-\widetilde{\Delta}_h}}(\Omega_1(u)(s)+\Omega_2 u(s))ds\right\|_{H^{-1/2}(\Man_h)}\lesssim \|\an{x}^{1/2}(B^i\partial_iu +\Omega_2 u)\|_{L^2_tL^2(\Man_h)}.
\]
Then by \eqref{loc-3} and Lemma \ref{lem:norm}, we have
\[
\|\an{x}^{1/2}\Omega_2 u\|_{L^2_tL^2(\Man_h)}\lesssim \|\an{x}^{2+}\Omega_2\|_{L^\infty}\|\an{x}^{-3/2-}u\|_{L^2_tL^2(\Man_h)}\lesssim \|\D_m u_0\|_{L^2(\Man_h)}\lesssim \|u_0\|_{H^1(\Man_h)}.
\]

Then by \eqref{loc-3}, Remark \ref{rem:fBR} and Lemma \ref{lem:norm}, we have
\[
\|\an{x}^{1/2}\Omega_1(u)\|_{L^2_tL^2(\Man_h)}\lesssim \|\an{x}^{1+}B\|_{L^\infty_x}\|\an{x}^{-1/2-}\widetilde{\nabla} u\|_{L^2_tL^2(\Man_h)}\lesssim \|\D u_0\|_{L^2}\lesssim \|u_0\|_{H^1(\Man_h)},
\]
so that 
\[
\|u\|_{L^q_t\dot{H}^{1/2-s}_r(\Man_h)}\lesssim \|u_0\|_{H^{1/2}(\Man_h)}+\|\D_m u_0\|_{H^{-1/2}(\Man_h)}+\|u_0\|_{H^1(\Man_h)}\lesssim \|u_0\|_{H^1(\Man_h)}.
\]
and this concludes the proof.
\end{proof}

\appendix
\section{Norm estimate}
Here we give some results concerning the relationship between the standard Sobolev norm and the one induced by the Dirac operator, which is used in the proof of Theorem \ref{teo1}.
\begin{lemma}\label{lem:norm}
Under assumptions \textbf{(A)}, for  $m\geq 0$,
\[
\|(m^2-\widetilde{\Delta}_h)^{1/2}u\|_{L^2(\Man_h)}\lesssim \|\D_m u\|_{L^2(\Man_h)}\lesssim \|(m^2-\widetilde{\Delta}_h)^{1/2}u\|_{L^2(\Man_h)}.
\]
\end{lemma}
\begin{proof}

Unsing that $\tilde\Delta_h$ is self-adjoint on $L^2(\Man_h)$, we have 
\[
\|(m^2 - \tilde \Delta_h)^{1/2}u\|_{L^2(\Man_h)}^2 = \an{(m^2 - \tilde \Delta_h)u,u}_{L^2(\Man_h)}.
\]
We also have
\[
-\an{\tilde \Delta_h u,u}_{L^2(\Man_h)} = h^{ij} \an{\partial_i u,\partial_j u}_{L^2(\Man_h)},
\]
so that
\[
-\an{\tilde \Delta_h u,u}_{L^2(\Man_h)} = \delta^{ij} \an{\partial_i u,\partial_j u }_{L^2(\Man_h)} +  (h^{ij} - \delta^{ij}) \an{\partial_i u,\partial_j u}_{L^2(\Man_h)}.
\]
Using Cauchy-Schwarz inequality and the fact that $|h^{ij} - \delta^{ij}|\ll 1$ yields
\[
\delta^{ij} \an{\partial_i u,\partial_j u}_{L^2(\Man_h)} \lesssim -\an{\tilde \Delta_h u,u}_{L^2(\Man_h)}\lesssim \delta^{ij} \an{\partial_i u,\partial_j u}_{L^2(\Man_h)}.
\]
As the Dirac operator is self-adjoint on $L^2(\Man_h)$, we have
\[
\|\mathcal D_m u\|_{L^2(\Man_h)}^2 = \an{\mathcal D_m^2 u, u}_{L^2(\Man_h)},
\]
and thanks to the identity $\mathcal D_m^2 = m^2 - \tilde \Delta_h + \Omega_1 + \Omega_2$ with $\Omega_1,\;\Omega_2$ given by \eqref{Omega} we get
\[
\|\mathcal D_m u\|_{L^2(\Man_h)}^2 = \an{(m^2 - \tilde \Delta_h) u, u}_{L^2(\Man_h)} + \an{(\Omega_1(u)  + \Omega_2u) ,u}_{L^2(\Man_h)}.
\]
Using the skew-symmetry of $B_i$, we get
\[
\an{B^i \partial_i u, u}_{L^2(\Man_h)} = -\an{\partial_i u, B^i u}_{L^2(\Man_h)},\quad \an{B_i \partial^i u, u}_{L^2(\Man_h)} = -\an{\partial^i u, B_i u}_{L^2(\Man_h)}
\]
By Cauchy-Schwarz inequality again we get
\[
|\an{\Omega_1(u), u}_{L^2(\Man_h)}| \leq  \sum_i \|\partial_i u\|_{L^2(\Man_h)} \|B\an{x}\|_{L^\infty}\|\an{x}^{-1} u\|_{L^2(\Man_h)}.
\]
Now, as the $L^2(\Man_h)$ and the $L^2(\R^3)$ norms are equivalent (due to the assumptions on $h$), we use Hardy inequality to get
\[
|\an{\Omega_1(u), u}_{L^2(\Man_h)}| \lesssim \|B\an{x}\|_{L^\infty}  \|(m^2 - \tilde \Delta_h)^{1/2} u\|_{L^2(\Man_h)}^2 .
\]
Similarly,
\[
|\an{\Omega_2 u,u}_{L^2(\Man_h)}|\lesssim \|\an{x}^2 \Omega_2\|_{L^\infty}\|(m^2 - \tilde \Delta_h)^{1/2} u\|_{L^2(\Man_h)}^2.
\]
Now, as $\|\an x B\|_{L^\infty}\ll 1$ (see Proposition \ref{propbound}) and 
\[
\|\an{x}^2\Omega_2\|_{L^\infty}\leq \|\an{x}^2(\partial^i B_i + B^iB_i -\Gamma^{j\; i}_{\; i}B_j- \frac14\mathcal R_h)\|_{L^\infty} \ll 1,
\]
we finally get that
\[
\|(m^2 - \tilde \Delta_h)^{1/2} u\|_{L^2(\Man_h)}^2\lesssim \|\mathcal D_m u\|_{L^2(\Man_h)}^2 \lesssim \|(m^2 - \tilde \Delta_h)^{1/2} u\|_{L^2(\Man_h)}^2
\]
and this concludes the proof.

\end{proof}

\medskip

\bibliographystyle{plain}
\bibliography{StrichartzAsFlat}

\end{document}